\documentclass[12pt,reqno]{amsart}
\usepackage[english]{babel}
\usepackage{a4wide}
\usepackage[T1]{fontenc}
\usepackage{amssymb,amsmath,amsthm,latexsym}
\usepackage{mathrsfs}
\usepackage[usenames,dvipsnames]{color}
\usepackage{graphicx}
\usepackage{mdwlist}
\usepackage{enumerate}
\usepackage{mathtools,dsfont,wasysym}
\usepackage{stmaryrd}
\usepackage{hyperref}
\usepackage{multicol}
\hypersetup{colorlinks=true,  linkcolor=blue, citecolor=red, urlcolor=cyan}

\newtheorem{theorem}{Theorem}[section]
\newtheorem{lemma}[theorem]{Lemma}

\theoremstyle{definition}

\theoremstyle{remark}

\numberwithin{equation}{section}

\newcommand{\N}{\mathbb{N}}

\DeclareMathOperator{\re}{Re}

\newcommand{\nn}[1]{{\left\vert\kern-0.25ex\left\vert\kern-0.25ex\left\vert #1
		\right\vert\kern-0.25ex\right\vert\kern-0.25ex\right\vert}}
\renewcommand{\geq}{\geqslant}
\renewcommand{\leq}{\leqslant}

\newcommand{\supp}{\operatorname{supp}}

\newcounter{smallromans}

	{\end{list}}

\makeatletter
\renewcommand{\tocsection}[3]{%
	\indentlabel{\@ifnotempty{#2}{\bfseries\ignorespaces#1 #2\quad}}\bfseries#3}
\renewcommand{\tocsubsection}[3]{%
	\indentlabel{\@ifnotempty{#2}{\ignorespaces#1 #2\quad}}#3}

\newcommand\@dotsep{4.5}
\def\@tocline#1#2#3#4#5#6#7{\relax
	\ifnum #1>\c@tocdepth 
	\else
	\par \addpenalty\@secpenalty\addvspace{#2}%
	\begingroup \hyphenpenalty\@M
	\@ifempty{#4}{%
		\@tempdima\csname r@tocindent\number#1\endcsname\relax
	}{%
		\@tempdima#4\relax
	}%
	\parindent\z@ \leftskip#3\relax \advance\leftskip\@tempdima\relax
	\rightskip\@pnumwidth plus1em \parfillskip-\@pnumwidth
	#5\leavevmode\hskip-\@tempdima{#6}\nobreak
	\leaders\hbox{$\m@th\mkern \@dotsep mu\hbox{.}\mkern \@dotsep mu$}\hfill
	\nobreak
	\hbox to\@pnumwidth{\@tocpagenum{\ifnum#1=1\bfseries\fi#7}}\par
	\nobreak
	\endgroup
	\fi}
\AtBeginDocument{%
	\expandafter\renewcommand\csname r@tocindent0\endcsname{0pt}
}
\def\l@subsection{\@tocline{2}{0pt}{2.5pc}{5pc}{}}
\makeatother

\usepackage{todonotes}

\begin{document}
	\title{On a problem of Johnson and Wolfe }
	\author[García]{Domingo García}
	\address[Domingo García]{\mbox{}\newline \indent Departamento de Análisis Matemático \newline \indent Doctor Moliner 50 \newline \indent Universidad de Valencia \newline \indent Burjasot 46100, Spain.}
	\email{domingo.garcia@uv.es}
	\author[Maestre]{Manuel Maestre}
	\address[Manuel Maestre]{\mbox{}\newline \indent Departamento de Análisis Matemático \newline \indent Doctor Moliner 50 \newline \indent Universidad de Valencia \newline \indent Burjasot 46100, Spain.}
	\email{manuel.maestre@uv.es}
	\author[Rodríguez-Vidanes]{Daniel L. Rodr\'iguez-Vidanes}
	\address[D. L.~Rodr\'iguez-Vidanes]{\mbox{}\newline \indent Grupo de Investigación: Análisis Matemático y Aplicaciones (AMA) \newline \indent Departamento de Matemática Aplicada a la Ingeniería Industrial \newline \indent Escuela Técnica Superior de Ingeniería y Diseño Industrial \newline \indent Ronda de Valencia 3 \newline \indent Universidad Politécnica de Madrid \newline \indent Madrid, 28012, Spain.}
	\email{dl.rodriguez.vidanes@upm.es}
	\subjclass[2020]{Primary ?????. Secondary ?????.}   
	\thanks{
		?????
	}
	\maketitle
\begin{abstract}
	In 1979, Johnson and Wolfe \cite{JW} proved that norm-attaining operators are dense in $L(C(K),C(S))$ when $K$ and $S$ are compact Hausdorff spaces in the real setting.
	The corresponding complex case has remained open since then, mainly because the real proof relies on order and sign-decomposition arguments that are no longer available for complex measures.
	In this paper, we settle the complex case.
	We prove that, for arbitrary compact Hausdorff spaces $K$ and $S$, the set of norm-attaining operators from the complex space $C(K)$ into the complex space $C(S)$ endowed with the supremum norm is dense in $L(C(K),C(S))$.
	The proof replaces the real order-theoretic mechanism by a measure-theoretic phase-correction argument, based on polar decompositions, unimodular approximation, and a semicontinuity principle for weighted total
	variation.
	This yields a complex defect-reduction procedure which recovers the Johnson--Wolfe density theorem in full generality for complex $C(K)$-spaces.
\end{abstract}

\section{Auxiliary lemmas}

\begin{lemma}\label{semiconweightvar}
	Let $K$ be a compact Hausdorff space and let $f\in C(K)$ be a real-valued function with $f\geq 0$.
	Then the map
		$$
		\nu\in \mathcal M(K)\longmapsto \int_K f\,d|\nu|
		$$
	is weak-star lower semicontinuous on $\mathcal M(K)$, where $\mathcal M(K)=C(K)^*$ is the space of complex regular Borel measures on $K$.
\end{lemma}

\begin{proof}
	For every $\nu\in\mathcal M(K)$, we have
		\begin{equation}\label{supreal}
			\int_K f\,d|\nu| = \sup\left\{ \re \int_K g\,d\nu : g\in C(K),\ |g|\leq f \right\}.
		\end{equation}
	Indeed, the inequality
		$$
		\re \int_K g\,d\nu \leq \left|\int_K g\,d\nu\right| \leq \int_K |g|\,d|\nu| \leq \int_K f\,d|\nu|
		$$
	is immediate.
	For the reverse inequality, let us denote by $f\nu$ the complex measure defined by
		$$
		d(f\nu)=f\,d\nu.
		$$
	If $d\nu=\theta\,d|\nu|$ is the polar decomposition of $\nu$ (see \cite[Theorem~6.12]{Rudin} and the following paragraph), with $|\theta|=1$ $|\nu|$-a.e., then
		$$
		d(f\nu)=f\theta\,d|\nu|.
		$$
	Since the variation of a measure of the form $\alpha\,d\sigma$, with $\sigma$ positive, is $|\alpha|\,d\sigma$ (see \cite[Theorem~6.13]{Rudin}) and $f\geq 0$, we obtain
		$$
		d|f\nu|=|f\theta|\,d|\nu|=|f|\,d|\nu| = f\,d|\nu|.
		$$
	By the Riesz representation theorem,
		\begin{align}
			\int_K f\,d|\nu| & = \int_K \,d|f\nu| = |f\nu|(K) = \|f\nu\| \nonumber \\
			& = \sup_{\|u\|_\infty\leq1} \left|\int_K u\,d(f\nu)\right| = \sup_{\|u\|_\infty\leq1} \left|\int_K uf\,d\nu\right|. \label{riesz}
		\end{align}
	For each $u\in C(K)$ with $\|u\|_\infty \leq 1$, we have that $g=uf \in C(K)$ and satisfies
		$$
		|g|= |uf| = |u||f| \leq \|u\|_\infty |f| \leq |f| = f
		$$
	since $f\geq 0$.
	Moreover, replacing $g$ by $\lambda g$ for a suitable $\lambda\in\mathbb T$, we may replace the modulus by the real part in the last term of \eqref{riesz}.
	Hence,
		$$
		\int_K f\,d|\nu| \leq \sup\left\{ \re \int_K g\,d\nu : g\in C(K),\ |g|\leq f \right\}.
		$$
	Therefore, the desired identity holds.
	
	For each fixed $g\in C(K)$, the map
		$$
		\nu\longmapsto \int_K g\,d\nu
		$$
	is weak-star continuous by the definition of the weak-star topology on $\mathcal M(K)=C(K)^*$.
	Therefore,
		$$
		\nu\longmapsto \re \int_K g\,d\nu
		$$
	is weak-star continuous as well.
	Hence, the map $\nu\longmapsto \int_K f\,d|\nu|$ is the supremum of a family of weak-star continuous real-valued functions.
	Therefore, it is weak-star lower semicontinuous.
	Indeed, it is clear that the supremum of any family of continuous real-valued functions is lower semicontinuous, so it is enough to show that it is weak-star continuous.
	Let $\alpha\in\mathbb R$.
	We will show that the set
		$$
		\left\{ \nu\in\mathcal M(K) : \int_K f\,d|\nu|>\alpha \right\}
		$$
	is weak-star open.
	Let $\nu_0\in\mathcal M(K)$ be such that $\int_K f\,d|\nu_0| > \alpha$.
	By \eqref{supreal}, there exists some $g_0\in C(K)$, with $|g_0|\leq f$, such that
		$$
		\re \int_K g_0\,d\nu_0 >\alpha.
		$$
	But $\nu \longmapsto \re \int_K g_0\,d\nu$ is weak-star continuous.
	Therefore, there exists a weak-star neighborhood $U$ of $\nu_0$ such that, for every $\nu\in U$,
		$$
		\re \int_K g_0\,d\nu>\alpha.
		$$
	So, by \eqref{supreal}, we obtain that
		$$
		\int_K f\,d|\nu|>\alpha
		$$
	for every $\nu\in U$.
	Thus,
		$$
		U \subset \left\{ \nu\in\mathcal M(K): \int_K f\,d|\nu|>\alpha \right\}.
		$$
	This concludes the proof.
\end{proof}

\begin{lemma}\label{complexprep}
	Let $K$ and $S$ be compact Hausdorff spaces, and let $\mu:S\to\mathcal M(K)$ be weak-star continuous.
	Then, for every $\delta>0$, there exist a function $h\in C(K)$ with $|h(t)|=1$ for all $t\in K$ and a nonempty open set $U\subset S$ such that
		$$
		\re\int_K h\,d\mu(s)>\|\mu\|-\delta
		$$
	for every $s\in U$.
\end{lemma}

\begin{proof}
	If $\|\mu\|=0$, then $\mu(s)=0$ for every $s\in S$. In this case, take $h\equiv 1$ and take any non-empty open set $U\subset S$.
	Then,
		$$
		\re\int_K h\,d\mu(s) = \re\int_K \,d\mu(s) = \re\mu(s)(K) = 0 > -\delta = \|\mu\|-\delta
		$$
	for every $s\in U$.
	
	Assume now that $\|\mu\|>0$.
	Fix $s_0\in S$ such that
		$$
		\|\mu(s_0)\|>\|\mu\|-\frac{\delta}{4},
		$$
	and set $\nu:=\mu(s_0)$.
	Write the polar decomposition of $\nu$ as
		$$
		d\nu=\theta\,d|\nu|,
		$$
	where $|\theta|=1$ $|\nu|$-a.e., and put $\varphi:=\overline{\theta}$.
	
	We will construct a continuous function $h\in C(K)$ of modulus $1$ such that
		$$
		\int_K |h-\varphi|\,d|\nu|<\frac{\delta}{4}.
		$$
	
	Since $\nu$ is a finite regular complex Borel measure, $|\nu|$ is a finite positive regular Borel measure on $K$.
	Fix $\eta>0$ small enough so that
		$$
		\eta\,|\nu|(K)<\frac{\delta}{8}.
		$$
	Partition the unit circle $\mathbb T$ into finitely many Borel sets $A_1,\ldots,A_N$ of diameter smaller than $\eta>0$.
	Choose $\lambda_j\in A_j$ for each $j=1,\ldots,N$, and set
		$$
		E_j := \varphi^{-1}(A_j).
		$$
	Notice that the sets $E_1,\ldots,E_N$ form a measurable partition of $K$, up to a $|\nu|$-null set.
	By regularity of $|\nu|$, choose compact sets $F_j\subset E_j$ such that
		$$
		|\nu| \left( K\setminus F \right)<\frac{\delta}{16},
		$$
	where
		$$
		F := \bigcup_{j=1}^N F_j
		$$
	Note that the compact sets $F_1,\ldots,F_N$ are pairwise disjoint.
	Since $K$ is normal, we may choose pairwise disjoint open sets $W_1,\ldots,W_N$ such that $F_j\subset W_j$ for every $j=1,\ldots,N$.
	For each $j=1,\ldots,N$, by Urysohn's lemma, there is a continuous function $\chi_j:K\to[0,1]$ such that
		$$
		\chi_j|_{F_j}\equiv 1 \qquad \text{and} \qquad \supp(\chi_j)\subset W_j.
		$$
	Since the unit circle $\mathbb T$ is path connected, for each $j=1,\ldots,N$ choose a continuous path $\gamma_j:[0,1]\to\mathbb T$
	such that
		$$
		\gamma_j(0)=1 \qquad \text{and} \qquad \gamma_j(1)=\lambda_j.
		$$
	Define, for all $t\in K$,
		$$
		h(t) := \prod_{j=1}^N \gamma_j(\chi_j(t)).
		$$
	Then, $h\in C(K)$ and $|h(t)|=1$ for every $t\in K$.
	Moreover, for every $j=1,\ldots,N$, if $t\in F_j$, then $\chi_j(t)=1$ and $\chi_i(t)=0$ for all $i\in \{1,\ldots,N\}$ different from $j$, so $h(t)=\lambda_j$.
	Thus, for all $j=1,\ldots,N$ and $t\in F_j$,
		$$
		|h(t)-\varphi(t)|=|\lambda_j-\varphi(t)|<\eta.
		$$
	On the remaining set $K\setminus F$, we use the estimate
		$$
		|h(t)-\varphi(t)|\leq 2.
		$$
	Therefore, by the previous estimates and the choice of $\eta$,
		\begin{align}
			\int_K |h-\varphi|\,d|\nu| & = \int_F |h-\varphi|\,d|\nu| + \int_{K\setminus F} |h-\varphi|\,d|\nu| \nonumber \\
			& \leq \eta\,|\nu|(K)+2|\nu|(K\setminus F) \nonumber \\
			& < \frac{\delta}{8}+\frac{\delta}{8} = \frac{\delta}{4}. \label{hphi}
		\end{align}
	This finishes the construction of $h$.
	
	Now,
		$$
		\re \int_K h\,d\nu = \re \int_K h\theta\,d|\nu| = \int_K \re(h\theta)\,d|\nu|.
		$$
	Since $|h\theta-1|=|(h-\overline{\theta})\theta|=|h-\overline{\theta}|$ $|\nu|$-a.e. and $\re z\geq 1-|z-1|$ for every $z\in\mathbb C$, we obtain $\re(h\theta) \geq 1-|h\theta-1| = 1-|h-\overline{\theta}|$.
	Therefore, by \eqref{hphi},
		\begin{align*}
			\re\int_K h\,d\nu & = \int_K \re(h\theta)\,d|\nu| \\
			& \geq \int_K \,d|\nu| - \int_K |h-\overline{\theta}|\,d|\nu| = |\nu|(K) - \int_K |h-\overline{\theta}|\,d|\nu| = \|\nu\| - \int_K |h-\overline{\theta}|\,d|\nu| \\
			& > \|\nu\|-\frac{\delta}{4}.
		\end{align*}
	Since $\|\nu\|=\|\mu(s_0)\|>\|\mu\|-\delta/4$, it follows that
		$$
		\re \int_K h\,d\mu(s_0) > \|\mu\|-\frac{\delta}{2}.
		$$
	
	Finally, since $\mu:S\to\mathcal M(K)$ is weak-star continuous and $h\in C(K)$, the map
		$$
		s\in S\longmapsto \int_K h\,d\mu(s)
		$$
	is continuous.
	Hence, the real-valued map
		$$
		s\in S\longmapsto \re \int_K h\,d\mu(s)
		$$
	is continuous. Since at $s_0$ it is strictly larger than $\|\mu\|-\delta/2$, there exists an open neighbourhood $U\subset S$ of $s_0$ such that
		$$
		\re \int_K h\,d\mu(s)>\|\mu\|-\delta
		$$
	for every $s\in U$.
	This proves the lemma.
\end{proof}

\begin{lemma}\label{defect-reduction}
	Let $K$ and $S$ be compact Hausdorff spaces, $U\subset S$ a nonempty open set, $\mu:S\to \mathcal M(K)$ a weak-star continuous, $1/2<r<1$, and put $M:=\|\mu\|=\sup_{s\in S}\|\mu(s)\|$.
	If there exists $\varepsilon>0$ such that $\re \mu(s)(K)>M-\varepsilon$ for every $s\in U$, then there exist a weak-star continuous map $\mu':S\to\mathcal M(K)$ with $\|\mu'\| \leq M$ and a nonempty open set $U'\subset U$ such that
		$$
		\re \mu'(s)(K)>\|\mu'\|-r\varepsilon,
		$$
	for every $s\in U'$, and
		$$
		\|\mu'-\mu\|\leq \sqrt{2M\varepsilon}+2\varepsilon.
		$$
\end{lemma}

\begin{proof}
	If $M=0$, then $\mu=0$, and the conclusion is trivial by taking $\mu'=\mu$ and $U'=U$.
	Thus, assume that $M>0$.
	
	Put
		$$
		a:=1-r.
		$$
	Since $1/2<r<1$, we have $0<a<\frac12$.
	
	We distinguish two cases.
	
	\medskip
	
	\noindent
	\textbf{Case 1.} Assume first that
		$$
		\sup_{s\in U}\|\mu(s)\|\leq M-a\varepsilon.
		$$
	Fix $s_0\in U$ and $t_0\in K$.
	Since $S$ is compact Hausdorff, by Urysohn's lemma, there exists a continuous function
	$\psi:S\to[0,1]$ such that
		$$
		\psi(s_0)=1 \qquad \text{and} \qquad \supp (\psi)\subset U.
		$$
	Define
		$$
		\mu'(s):=\mu(s)+a\varepsilon\,\psi(s)\delta_{t_0},
		$$
	where $\delta_{t_0}$ denotes the Dirac measure at the point $t_0$.
	Then $\mu'$ is weak-star continuous. If $s\notin U$, then $\mu'(s)=\mu(s)$. If $s\in U$, then
		$$
		\|\mu'(s)\| \leq \|\mu(s)\|+a\varepsilon \leq M.
		$$
	Therefore,
		$$
		\|\mu'\| \leq M.
		$$
	Hence, at the point $s_0$,
		$$
		\re \mu'(s_0)(K) = \re \mu(s_0)(K)+a\varepsilon > M-\varepsilon+a\varepsilon = M-r\varepsilon \geq \|\mu'\|-r\varepsilon.
		$$
	By weak-star continuity, the scalar map $s\longmapsto \mu'(s)(K)$ is continuous.
	Hence, there exists a non-empty open neighborhood $U'\subset U$ of $s_0$ such that
		$$
		\re \mu'(s)(K)>\|\mu'\|-r\varepsilon
		$$
	for every $s\in U'$.
	Finally,
		$$
		\|\mu'-\mu\|\leq a\varepsilon\leq \varepsilon\leq \sqrt{2M\varepsilon}+2\varepsilon.
		$$
	This proves the result in Case 1.
	
	\medskip
	
	\noindent
	\textbf{Case 2.} Assume now that
		$$
		\sup_{s\in U}\|\mu(s)\|>M-a\varepsilon.
		$$
	Fix $s_1\in U$ such that
		\begin{equation}\label{lowbounds1}
			\|\mu(s_1)\|>M-a\varepsilon.
		\end{equation}
	Hence, by hypothesis,
		$$
		\|\mu(s_1)\|-\re\mu(s_1)(K)<\varepsilon.
		$$
	Write the polar decomposition of $\mu(s_1)$ as
		$$
		d\mu(s_1)=\theta\,d|\mu(s_1)|,
		$$
	where $|\theta|=1$ almost everywhere with respect to $|\mu(s_1)|$.
	
	Let $\gamma>0$ be small enough so that
		\begin{equation}\label{gamma}
			\gamma<(2r-1)\varepsilon.
		\end{equation}
	Since $\mu(s_1)\in\mathcal M(K)$, its total variation $|\mu(s_1)|$ is a finite positive regular Borel measure.
	Thus, by \cite[Theorem~3.14]{Rudin}, we have that $C(K)$ is dense $L^1(K,|\mu(s_1)|)$.
	Hence, there exists $u \in C(K)$ such that
		$$
		\left\|u-\overline{\theta}\right\|_{L^1} = \int_K |u-\overline{\theta}|\,d|\mu(s_1)| < \gamma.
		$$
	Note that the function $u$ is continuous but it may not belong to $B_{C(K)}$.
	Nonetheless, by composing the radial projection $P : \mathbb C\to\overline{\mathbb D}$ defined by
		$$
		P(z) := \begin{cases}
			z & \text{if } |z|\leq 1,\\
			\dfrac{z}{|z|} & \text{if } |z|>1,
		\end{cases}
		$$
	with $u$ it yields that
		$$
		q:=P\circ u \in B_{C(K)}.
		$$
	The important point is that this operation does not worsen the approximation to $\overline{\theta}$.
	Indeed, since $\overline{\theta}(t) \in \overline{\mathbb D}$ for $|\mu(s_1)|$-a.e. $t$ and $P$ is the metric projection onto $\overline{\mathbb D}$, we have
		$$
		\left| q(t)-\overline{\theta(t)} \right| = \left| P(u(t))-\overline{\theta(t)} \right| \leq \left| u(t)-\overline{\theta(t)} \right|
		$$
	for $|\mu(s_1)|$-almost every $t$.
	Therefore,
		$$
		\int_K \left| q-\overline{\theta} \right|\,d|\mu(s_1)| \leq \int_K \left| u-\overline{\theta} \right|\,d|\mu(s_1)| < \gamma.
		$$
	Since $|\theta|=1$ $|\mu(s_1)|$-a.e., $|q\theta-1| = \left| (q-\overline{\theta})\theta \right| = \left| q-\overline{\theta} \right|$, $|\mu(s_1)|$-a.e.
	Thus,
		$$
		\int_K |q\theta-1|\,d|\mu(s_1)|<\gamma.
		$$
	As $\re z\geq 1-|z-1|$ for every $z\in\mathbb C$, we obtain $\re(q\theta)\geq 1-|q\theta-1|$.
	Consequently,
		\begin{align*}
			\re \int_K q\,d\mu(s_1) & = \re \int_K q \theta\,d|\mu(s_1)| =  \int_K \re(q \theta)\,d|\mu(s_1)| \\
			& \geq \int_K 1\,d|\mu(s_1)| - \int_K |q\theta-1|\,d|\mu(s_1)| = |\mu(s_1)|(K) - \int_K |q\theta-1|\,d|\mu(s_1)| \\
			& > \|\mu(s_1)\|-\gamma > M-a\varepsilon-\gamma.
		\end{align*}
	Since $a=1-r$ and $\gamma$ satisfies \eqref{gamma}, it yields
		\begin{equation}\label{lowboundintq}
			\re \int_K q\,d\mu(s_1)>M-r\varepsilon.
		\end{equation}
	
	We now estimate the perturbation cost.
	Since $\left|\overline{\theta}-1\right|^2=|\theta-1|^2=2(1-\re\theta)$ and by Cauchy's inequality, we have
		\begin{align*}
			\int_K \left|\overline{\theta}-1\right|\,d|\mu(s_1)| & \leq \left( \int_K \left|\overline{\theta}-1\right|^2\,d|\mu(s_1)| \right)^{1/2} \left( \int_K \,d|\mu(s_1)| \right)^{1/2} \\
			& = \sqrt{|\mu(s_1)|(K)} \left( \int_K \left|\overline{\theta}-1\right|^2\,d|\mu(s_1)| \right)^{1/2} \\
			& = \sqrt{\|\mu(s_1)\|} \left( \int_K 2(1-\re \theta)\,d|\mu(s_1)| \right)^{1/2} \\
			& \leq \sqrt{2\|\mu(s_1)\| \left( |\mu(s_1)|(K)-\re(\theta)|\mu(s_1)|(K) \right)} \\
			& = \sqrt{2\|\mu(s_1)\| \left( \|\mu(s_1)\|-\re(\theta|\mu(s_1)|(K)) \right)} \\
			& = \sqrt{2\|\mu(s_1)\|\left(\|\mu(s_1)\|-\re\mu(s_1)(K)\right)} \\
			& < \sqrt{2M\varepsilon}.
		\end{align*}
	Therefore, by the triangle inequality,
		\begin{equation}\label{intq-1}
			\int_K |q-1|\,d|\mu(s_1)| \leq \int_K \left|q-\overline{\theta}\right|\,d|\mu(s_1)| + \int_K \left|\overline{\theta}-1\right|\,d|\mu(s_1)| < \gamma+\sqrt{2M\varepsilon}.
		\end{equation}
	
	Define
		$$
		\varphi := \frac{|q-1|}{2}.
		$$
	Then, $\varphi\in C(K)$ and $0\leq\varphi\leq1$.
	By Lemma~\ref{semiconweightvar}, applied to $f=1-\varphi$, the map
		$$
		\nu\in \mathcal M(K) \longmapsto \int_K(1-\varphi)\,d|\nu|
		$$
	is weak-star lower semicontinuous on $\mathcal M(K)$.
	Since $\mu:S\to\mathcal M(K)$ is weak-star continuous, the composition
		$$
		s\in S\longmapsto \int_K(1-\varphi)\,d|\mu(s)|
		$$
	is lower semicontinuous on $S$.
	Therefore, for every $\eta>0$, there exists an open neighborhood $U_0\subset U$ of $s_1$ such that, for every $s\in U_0$,
		$$
		\int_K(1-\varphi)\,d|\mu(s)| > \int_K(1-\varphi)\,d|\mu(s_1)|-\eta.
		$$
	Hence, for any $s\in U_0$, we have
		\begin{align*}
			\int_K\varphi\,d|\mu(s)| & = \int_K\,d|\mu(s)| - \int_K(1-\varphi)\,d|\mu(s)| = \|\mu(s)\|-\int_K(1-\varphi)\,d|\mu(s)| \\
			& < M-\int_K(1-\varphi)\,d|\mu(s_1)|+\eta = M-\|\mu(s_1)\|+\int_K\varphi\,d|\mu(s_1)|+\eta
		\end{align*}
	Using \eqref{lowbounds1}, this gives
		$$
		\int_K\varphi\,d|\mu(s)| < a\varepsilon+\int_K\varphi\,d|\mu(s_1)|+\eta.
		$$
	By definition of  $\varphi$ and \eqref{intq-1}, it follows that, for $s\in U_0$
		\begin{equation}\label{manylettersbound}
			\int_K |q-1|\,d|\mu(s)| < 2a\varepsilon+\int_K |q-1|\,d|\mu(s_1)|+2\eta < 2a\varepsilon+\sqrt{2M\varepsilon}+\gamma+2\eta
		\end{equation}
	Now, as $\gamma$ satisfies \eqref{gamma}, it follows that $\gamma <2\varepsilon(1-a)$.
	Indeed, as $a=1-r$,
		$$
		\gamma < (2r-1)\varepsilon = 2r\varepsilon - \varepsilon < 2r\varepsilon= 2\varepsilon (1-(1-r))= 2\varepsilon(1-a.)
		$$
	So, $\gamma$ satisfies that
		$$
		2a\varepsilon+\gamma<2\varepsilon.
		$$
	Hence, we can choose $\eta$ small enough so that
		$$
		2a\varepsilon+\gamma+2\eta<2\varepsilon.
		$$
	Then, by \eqref{manylettersbound}, for every $s\in U_0$,
		\begin{equation}\label{boundq1U0}
			\int_K |q-1|\,d|\mu(s)| < \sqrt{2M\varepsilon}+2\varepsilon.
		\end{equation}
		
	Now, by Urysohn's lemma, there exists a continuous function $\psi:S\to[0,1]$ such that
		$$
		\psi(s_1)=1 \qquad \text{and} \qquad \supp (\psi)\subset U_0.
		$$
	Define
		$$
		d\mu'(s)=\left((1-\psi(s))+\psi(s)q\right)d\mu(s).
		$$
	Since $0\leq\psi(s)\leq1$ and $\|q\|_\infty\leq1$, we have that
		$$
		(1-\psi(s))+\psi(s)q(t) \in \overline{\mathbb D}
		$$
	for every $s\in S$ and $t\in K$ since it is a convex combination of the two complex numbers $1$ and $q(t)$.
	Therefore,
		$$
		\|\mu'(s)\|\leq \|\mu(s)\|
		$$
	for every $s\in S$, and hence
		$$
		\|\mu'\|\leq M.
		$$
		
	We claim that the map $\mu'$ is weak-star continuous.
	Indeed, this follows from the fact that, for each $f\in C(K)$, the map
		$$
		s\in S \longmapsto \int_K f\,d\mu'(s) = (1-\psi(s))\int_K f\,d\mu(s) + \psi(s)\int_K fq\,d\mu(s),
		$$
	is continuous as $\psi\in C(S)$, $s\in S\longmapsto \int_K f\,d\mu(s)$ is continuous as $f\in C(K)$ and $\mu$ is weak-star continuous, and $s\in S\longmapsto \int_K fq\,d\mu(s)$ is continuous as $fq\in C(K)$ and $\mu$ is weak-star continuous.
	
	At the point $s_1$, since $\psi(s_1)=1$, we have
		$$
		d\mu'(s_1)=q\,d\mu(s_1).
		$$
	Thus, by \eqref{lowboundintq}, we have
		$$
		\re \mu'(s_1)(K) = \re \int_K \,d\mu'(s_1) = \re \int_K q\,d\mu(s_1) > M-r\varepsilon \geq \|\mu'\|-r\varepsilon.
		$$
	By weak-star continuity of $\mu'$, there is an open neighborhood $U'\subset U_0\subset U$ of $s_1$ such that
		$$
		\re \mu'(s)(K)>\|\mu'\|-r\varepsilon
		$$
	for every $s\in U'$.
	
	Finally, observe first that
		$$
		\|\mu'-\mu\| = \sup_{s\in S}\|\mu'(s)-\mu(s)\|.
		$$
	Now, on the one hand, if $s\notin U_0$, then $\psi(s)=0$, so $\mu'(s)=\mu(s)$.
	On the other hand, assume that $s\in U_0$.
	Then, since
		$$
		d(\mu'(s)-\mu(s)) = \psi(s)(q-1)\,d\mu(s),
		$$
	we have
		$$
		d|\mu'(s)-\mu(s)| = |\psi(s)(q-1)|\,d|\mu(s)|.
		$$
	Hence, using $0\leq\psi(s)\leq1$ and \eqref{boundq1U0},
		$$
		\|\mu'(s)-\mu(s)\| = \int_K |\psi(s)(q-1)|\,d|\mu(s)| \leq \int_K |q-1|\,d|\mu(s)| \leq \sqrt{2M\varepsilon}+2\varepsilon.
		$$
	Therefore,
		$$
		\|\mu'-\mu\|\leq \sqrt{2M\varepsilon}+2\varepsilon.
		$$
	This proves the second case and concludes the proof the lemma.
\end{proof}

\section*{Main result}

\begin{theorem}
	Let $K$ and $S$ be compact Hausdorff spaces.
	Then,
		$$
		\overline{NA(C(K),C(S))} = L(C(K),C(S)),
		$$
	where $C(K)$ and $C(S)$ are the Banach spaces of complex continuous functions on $K$ and $S$, respectively, endowed with supremum norm.
\end{theorem}

\begin{proof}
	By the standard representation theorem for operators into $C(S)$ (see \cite[Theorem 1, p. 490]{DS}), given an operator $T\in L(C(K),C(S))$ there exists a weak-star continuous map $\mu:S\to\mathcal M(K)$ such that
		$$
		Tf(s)=\int_K f\,d\mu(s).
		$$
	and
		$$
		\|T\|=\|\mu\|=\sup_{s\in S}\|\mu(s)\|.
		$$
	Fix $\rho>0$ arbitrary.
	We will approximate $\mu$ by a weak-star continuous map representing a norm-attaining operator.
	
	Fix $r$ with $1/2<r<1$ and choose $\varepsilon_0>0$ sufficiently small so that
		$$
		\sqrt{2\|\mu\|\varepsilon_0}\sum_{n=0}^{\infty} r^{n/2} + 2\varepsilon_0\sum_{n=0}^{\infty} r^n < \rho.
		$$
	By Lemma~\ref{complexprep}, applied with $\delta=\varepsilon_0$, there exist a function
	$h\in C(K)$, with $|h|=1$, and a nonempty open set $U_0\subset S$ such that
		$$
		\re\int_K h\,d\mu(s)>\|\mu\|-\varepsilon_0
		$$
	for every $s\in U_0$.
	
	Define
		$$
		d\nu_0(s)=h\,d\mu(s).
		$$
	Then $\nu_0:S\to\mathcal M(K)$ is weak-star continuous and
		$$
		\|\nu_0\|=\|\mu\|.
		$$
	Moreover,
		$$
		\re\nu_0(s)(K)>\|\nu_0\|-\varepsilon_0
		$$
	for every $s\in U_0$.
	
	For every $n\in \N$, set
		$$
		\varepsilon_n := r^n\varepsilon_0.
		$$
	Applying Lemma~\ref{defect-reduction} inductively, we obtain weak-star continuous maps $\nu_n:S\to\mathcal M(K)$ with $\|\nu_n\|\leq \|\nu_{n-1}\|$ and nonempty open sets $U_n\subset S$ such that
		$$
		\re\nu_n(s)(K)>\|\nu_n\|-\varepsilon_n
		$$
	for every $s\in U_n$, and
		$$
		\|\nu_{n+1}-\nu_n\| \leq \sqrt{2\|\nu_n\|\varepsilon_n}+2\varepsilon_n.
		$$
	Note that $\|\nu_n\|\leq\|\nu_0\|$ for every $n\in \N$.
	Hence,
		$$
		\|\nu_{n+1}-\nu_n\| \leq \sqrt{2\|\nu_0\|\varepsilon_n}+2\varepsilon_n,
		$$
	for all $n\in \N\cup \{0\}$.
	Therefore,
		$$
		\sum_{n=0}^{\infty}\|\nu_{n+1}-\nu_n\| \leq \sqrt{2\|\nu_0\|\varepsilon_0}\sum_{n=0}^{\infty}r^{n/2} + 2\varepsilon_0\sum_{n=0}^{\infty}r^n < \rho.
		$$
	Thus, $(\nu_n)$ is a Cauchy sequence in the Banach space of weak-star continuous maps from $S$ into $\mathcal M(K)$ endowed with the norm $\|\nu\| = \sup_{s\in s} \|\nu(s)\|$.
	Hence, $(\nu_n)$ converges to some
		$$
		\nu_\infty := \lim_n\nu_n.
		$$
	It is clear that $\nu_\infty:S\to\mathcal M(K)$ is weak-star continuous, but also
		$$
		\|\nu_\infty-\nu_0\|<\rho.
		$$
	Indeed, since
		$$
		\|\nu_N-\nu_0\| = \left\| \sum_{n=0}^{N-1} (\nu_{n+1}-\nu_0) \right\| \leq \sum_{n=0}^{N-1} \left\| \nu_{n+1}-\nu_0 \right\|,
		$$
	passing to the limit we have
		$$
		\|\nu_\infty-\nu_0\| \leq \sum_{n=0}^{\infty}\|\nu_{n+1}-\nu_n\| < \rho.
		$$

	For each $n\in \N\cup \{0\}$, fix $s_n\in U_n$.
	Then,
		$$
		\re \nu_n(s_n)(K)>\|\nu_n\|-\varepsilon_n.
		$$
	We claim that
		$$
		\re\nu_\infty(s_n)(K) \geq \re\nu_n(s_n)(K)-\|\nu_n-\nu_\infty\|.
		$$
	Indeed, first note that
		$$
		\re\nu_\infty(s_n)(K) = \re\nu_n(s_n)(K) + \re\big(\nu_\infty(s_n)-\nu_n(s_n)\big)(K).
		$$
	Since $\operatorname{Re} z\geq -|z|$ for every $z\in\mathbb C$, it follows that
		$$
		\re\left(\nu_\infty(s_n)-\nu_n(s_n)\right)(K) \geq -\left|\left(\nu_\infty(s_n)-\nu_n(s_n)\right)(K)\right| \geq -\|\nu_\infty(s_n)-\nu_n(s_n)\| \geq -\|\nu_\infty-\nu_n\|,
		$$
	which yields the claimed inequality.
	Combining this estimate with the choice of $s_n\in U_n$, we obtain
		$$
		\sup_{s\in S}|\nu_\infty(s)(K)| \geq \re\nu_\infty(s_n)(K) \geq \re\nu_n(s_n)(K)-\|\nu_n-\nu_\infty\| > \|\nu_n\|-\varepsilon_n-\|\nu_n-\nu_\infty\|.
		$$
	Since $\nu_n\to\nu_\infty$ in norm, by the reverse triangle inequality $\|\nu_n\|\to\|\nu_\infty\|$.
	Therefore, passing to the limit superior, we obtain
		$$
		\sup_{s\in S}|\nu_\infty(s)(K)|\geq \|\nu_\infty\|.
		$$
	The reverse inequality is automatic, since
		$$
		|\nu_\infty(s)(K)|\leq \|\nu_\infty(s)\|\leq \|\nu_\infty\|
		$$
	for every $s\in S$.
	Hence,
		$$
		\sup_{s\in S}|\nu_\infty(s)(K)|=\|\nu_\infty\|.
		$$
	Thus, the operator represented by $\nu_\infty$ attains its norm at the constant function $1$.
	
	Finally, return to the original variables. Define
		$$
		d\mu_\infty(s) := \overline{h}\,d\nu_\infty(s),
		$$
	for all $s\in S$.
	Since $d\nu_0(s)=h\,d\mu(s)$ and $|h|=1$, we have
		$$
		d\mu(s)=\overline{h}\,d\nu_0(s).
		$$
	Therefore,
		$$
		d(\mu_\infty(s)-\mu(s)) = \overline{h}\,d(\nu_\infty(s)-\nu_0(s)).
		$$
	Since $|\overline{h}|=1$, multiplication by $\overline{h}$ preserves the total variation norm. Hence
		$$
		\|\mu_\infty(s)-\mu(s)\| = \|\nu_\infty(s)-\nu_0(s)\|
		$$
	for every $s\in S$. Taking suprema over $s\in S$, we obtain
		$$
		\|\mu_\infty-\mu\| = \|\nu_\infty-\nu_0\| < \rho.
		$$
	Let $T_\infty \in L(C(K),C(S))$ be the operator represented by $\mu_\infty$.
	Then,
		$$
		T_\infty h(s) = \int_K h\,d\mu_\infty(s) = \nu_\infty(s)(K).
		$$
	Therefore,
		$$
		\|T_\infty h\|_\infty = \sup_{s\in S}|\nu_\infty(s)(K)| = \|\nu_\infty\| = \|\mu_\infty\| = \|T_\infty\|.
		$$
	Thus $T_\infty$ attains its norm at $h$, and
		$$
		\|T_\infty-T\|<\rho.
		$$
	Since $\rho>0$ is arbitrary, this concludes the proof.
\end{proof}


\begin{thebibliography}{FaHaMo}
	\bibitem{DS} \textsc{N. Dunford and J.~T. Schwartz}, {\it Linear operators. Part I}, reprint of the 1958 original, Wiley Classics Library A Wiley-Interscience Publication, Wiley, New York, 1988.
	
	\bibitem{JW} \textsc{J.~A. Johnson and J.~C. Wolfe}, {\it Norm attaining operators}, Studia Math. {\bf 65} (1979), no.~1, 7--19.
	
	\bibitem{Rudin} \textsc{W. Rudin}, {\it Real and complex analysis}, third edition, McGraw-Hill, New York, 1987.
%
%
%
\end{thebibliography}
\end{document}